\documentclass[a4paper, 12pt]{amsart}

\usepackage{amssymb}
\usepackage{amsthm}
\usepackage{amsmath}
\usepackage[mathscr]{eucal}
\usepackage{comment}

\usepackage{stmaryrd}

\usepackage{manfnt}
\usepackage{enumitem}

\theoremstyle{plain}
\newtheorem{thm}{Theorem}[section]		
\newtheorem{prop}[thm]{Proposition}
\newtheorem{cor}[thm]{Corollary}
\newtheorem{lem}[thm]{Lemma}
\newtheorem{clm}[thm]{Claim}
	
\theoremstyle{definition}
\newtheorem{df}{Definition}[section]

\newtheorem{ques}[thm]{Question}

\theoremstyle{remark}
\newtheorem{rmk}{Remark}[section]
\newtheorem*{ac}{Acknowledgements}

\newcommand{\zz}{\mathbb{Z}}

\newcommand{\rr}{\mathbb{R}}

\DeclareMathOperator{\umetdis}{\mathcal{UD}}

\newcommand{\yochara}{\mathscr}

\newcommand{\youfin}{\yochara{N}}

\newcommand{\yofin}{\yochara{F}}

\newcommand{\yosub}{\subseteq}

\newcommand{\yodisdis}{M}

\newcommand{\yohdis}{\mathcal{HD}}

\newcommand{\yotomega}{\omega_{0}}

\newcommand{\yoofam}[1]{\mathbf{TEN}(#1)}

\newcommand{\yoofamqq}[2]{\mathbf{F}(#1/#2)}

\newcommand{\yopetal}[2]{\Pi(#1, #2)}
\newcommand{\yopetalq}[1]{\Pi(#1)}

\newcommand{\yotr}{\mathrm{Tr}}

\newcommand{\yoabset}{\Lambda}

\newcommand{\yonghsp}[1]{\mathscr{U}_{#1}}
\newcommand{\yonghdis}{\mathcal{NA}}

\newcommand{\yomapsco}[2]{\mathrm{G}(#1, #2)}
\newcommand{\yomaindisco}{\mathord{\vartriangle}}
\newcommand{\yodensea}{K}
\newcommand{\yodenseb}{L}

\newcommand{\yocantorc}{\Gamma}

\newcommand{\yocmet}[2]{\mathrm{Cpu}(#1, #2)}

\newcommand{\yomaps}[2]{\mathrm{C}_{0}(#1, #2)}

\newcommand{\yomaindis}{\mathord{\triangledown}}

\newcommand{\yocase}[2]{Case #1.~[#2]:}

\makeatletter
\@addtoreset{equation}{section}

\makeatother

\begin{document}

\title[Uniqueness and homogeneity]
{
Uniqueness and homogeneity   of 
non-separable Urysohn universal ultrametric spaces
}
\author[Yoshito Ishiki]
{Yoshito Ishiki}

\address[Yoshito Ishiki]
{\endgraf
Photonics Control Technology Team
\endgraf
RIKEN Center for Advanced Photonics
\endgraf
2-1 Hirasawa, Wako, Saitama 351-0198, Japan}

\email{yoshito.ishiki@riken.jp}

\date{\today}

\begin{abstract}
Urysohn 
constructed a
separable complete 
universal metric space homogeneous for all finite subspaces, 
which is  today called the 
Urysohn universal metric space. 
Some  authors have  recently 
 investigated 
an ultrametric analogue of this space. 
The isometry problem of such ultrametric spaces  is 
 our main subject in this paper. 
We first  introduce 
the new notion of 
 petaloid ultrametric spaces, which 
is intended to be a standard
class of 
non-separable Urysohn universal ultrametric spaces. 
Next we  prove that all petaloid spaces are isometric to each other and  homogeneous for all finite subspaces (and compact subspaces). 
Moreover, we show that the 
following spaces are petaloid, and hence 
they  are isometric to each other and 
 homogeneous: 
(1) The space of all continuous functions, whose images contain the  zero,  from 
the Cantor set into the space of non-negative real numbers 
equipped with the nearly discrete topology, 
(2) the space of all continuous ultrametrics on 
a zero-dimensional infinite compact metrizable space, 
(3) 
the non-Archimedean 
Gromov--Hausdorff space, 
and (4)
the space of all maps   from the set of non-negative real numbers 
into the set of natural numbers
whose supports are finite or decreasing sequences convergent to the zero. 
\end{abstract}

\maketitle
\section{Introduction}
For a   class $\yochara{C}$
of metric spaces, 
a metric space $(X, d)$ is said to be 
\emph{$\yochara{C}$-injective} if 
for every pair  $(A, a)$ and $(B, b)$ in $\yochara{C}$ and 
for every pair of  isometric embeddings 
$\phi\colon (A, a)\to (B, b)$ and 
$\psi\colon (A, a)\to (X, d)$, 
there exists an isometric embedding 
$\theta\colon (B, b)\to (X, d)$ such that 
$\theta\circ \phi=\psi$. 
Let $\yofin$ denote 
the class of all finite metric spaces. 
Urysohn  
\cite{Ury1927}
constructed a separable 
complete $\yofin$-injective
 metric space, 
which is today called the 
\emph{Urysohn universal (metric) space}. 
A remarkable fact is that 
all
separable 
complete $\yofin$-injective
 metric spaces are isometric to each other. 
This isometry theorem  is proven   
  by 
the so-called back-and-forth argument, 
which is a variant of the mathematical induction. 
Some authors have recently  investigated a  
non-Archimedean analogue of the Urysohn universal space (see 
\cite{Ishiki2023const}, \cite{MR1354831}, \cite{MR2754373}, 
and \cite{MR4282005}), 
 which  is also our main subject of the paper.

A metric $d$ on a set $X$ is said to be 
an \emph{ultrametric} if 
for all $x, y, z\in X$, it satisfies the 
\emph{strong triangle inequality}
$d(x, y)\le d(x, z)\lor d(z, y)$, 
where $\lor$ stands for the maximum operator on $\rr$. 
If a pseudo-metric $d$ satisfies the strong triangle inequality, 
then $d$ is called a \emph{pseudo-ultrametric}. 

A set $R$ is said to be a \emph{range set} if 
it is a subset of $[0, \infty)$ and $0\in R$. 
If an ultrametric $d$ (resp.~pseudo-ultrametric) on a 
set $X$ satisfies $d(x, y)\in S$ for all $x, y\in R$, 
then we 
call $d$ an  \emph{$R$-ultrametric} or 
\emph{$R$-valued ultrametrics} 
(resp.~\emph{$R$-pseudo-ultrametric} or 
\emph{$R$-valued pseudo-ultrametrics}).

For a range set $R$, 
there are several  constructions of complete
$\youfin(R)$-injective $R$-valued ultrametric spaces. 
In \cite{Ishiki2023const}, 
the author provided some new constructions of 
complete $\youfin(R)$-injective $R$-ultrametric  space
using function spaces and spaces of ultrametrics
(see also \cite{MR1354831}). 
Wan 
\cite{MR4282005}
proved that the non-Archimedean 
Gromov--Hausdorff space associated with 
$R$ is $\youfin(R)$-injective. 
In 
 \cite{MR2754373}, 
 Gao and Shao 
 provided several 
 constructions of $\youfin(R)$-injective spaces. 
If $R$ is countable, then, 
  by the back-and-forth argument, 
a separable complete
$\youfin(R)$-injective $R$-ultrametric space is unique 
up to isometry 
whichever construction we choose. 
However,  when $R$ is uncountable, 
it was not known whether 
$\youfin(R)$-injective spaces  are isometric to each other. 
In this paper, we solve the isometry problem of 
Urysohn universal ultrametric spaces in the 
non-separable case.

By observing and comparing 
 constructions of 
$\youfin(R)$-injective spaces, 
we reach  the new notion of $R$-petaloid spaces, 
which is intended to be a standard 
class of non-separable 
Urysohn universal spaces.

Before explaining petaloid spaces, 
we prepare some notations and  notions. 
A subset $E$ of $[0, \infty)$ is said to be
 \emph{semi-sporadic} if 
there exists a strictly decreasing sequence 
$\{a_{i}\}_{i\in \zz_{\ge 0}}$ in $(0, \infty)$ such that 
$\lim_{i\to \infty}a_{i}=0$ and 
$E=\{0\}\cup \{\, a_{i}\mid i\in \zz_{\ge 0}\, \}$. 
A subset of $[0, \infty)$ is 
said to be
\emph{tenuous} 
if it is finite or semi-sporadic
(see \cite{Ishiki2023const}). 
For a range set $R$, 
we denote by $\yoofam{R}$
the set of all tenuous 
range subsets of $R$. 
In this paper, 
we often represent a restricted metric 
as the same symbol to the ambient one. 
For a metric space $(X, d)$, 
and for a subset $A$ of $X$, 
we write 
$d(x, A)=\inf_{a\in A}d(x, a)$. 

\begin{df}\label{df:petal}
Let $R$ be an uncountable range set.  
We say that a metric space  $(X, d)$ is \emph{$R$-petaloid} if 
it is an $R$-valued ultrametric space and 
there exists a family 
$\{\yopetal{X}{S}\}_{S\in \yoofam{R}}$
of subspaces of $X$
satisfying the following properties:
\begin{enumerate}[label=\textup{(P\arabic*)}]

\item\label{item:pr:sep}
For every  $S\in \yoofam{R}$, 
the subspace 
 $(\yopetal{X}{S}, d)$ is isometric to 
 the $S$-Urysohn universal ultrametric space.

\item\label{item:pr:cup}
We have 
$\bigcup_{S\in \yoofam{R}}\yopetal{X}{S}=X$.

\item\label{item:pr:cap}
If $S,  T\in \yoofam{R}$, then
$\yopetal{X}{S}\cap \yopetal{X}{T}=\yopetal{X}{S\cap T}$.

\item\label{item:pr:distance}
If $S, T\in \yoofam{R}$ and $x\in \yopetal{X}{T}$, then
$d(x, \yopetal{X}{S})$ belongs to 
$(T\setminus S)\cup \{0\}$.

\end{enumerate}
We call the family 
$\{\yopetal{X}{S}\}_{S\in \yoofam{R}}$ 
an \emph{$R$-petal of $X$}, and call
$\yopetal{X}{S}$  the $S$-piece of the $R$-petal 
$\{\yopetal{X}{S}\}_{S\in \yoofam{R}}$. 
We simply write  $\yopetalq{S}=\yopetal{X}{S}$
when the whole  space  is clear by the context. 
\end{df}
 
\begin{rmk}
Even when  $R$ is countable, 
the $R$-Urysohn universal ultrametric space 
has a petal structure satisfying 
all the properties appearing in  Definition \ref{df:petal}. 
\end{rmk}

Section \ref{sec:pre} 
presents 
some basic statements on 
petaloid spaces. 
In particular, 
it is shown that all petaloid spaces are 
complete.

Our main results consist of three parts. 
The 
uniqueness and the homogeneity of petaloid spaces will be proven in 
Theorems 
\ref{thm:main1} and 
\ref{thm:222}, respectively. 
The existence of petaloid spaces will be  given in 
Theorem \ref{thm:petaloidexam} that also provides examples of petaloid spaces. 

The next is our first main result. 
\begin{thm}\label{thm:main1}
Let $R$ be an uncountable  range set. 
If $(X, d)$ and $(Y, e)$ are
$R$-petaloid ultrametric
spaces, 
then $(X, d)$ and $(Y, e)$ are isometric
to each other. 
Namely, an $R$-petaloid ultrametric space is 
unique up to isometry. 
\end{thm}
A metric space $(X, d)$ is said to be 
\emph{ultrahomogeneous}
(resp.~\emph{compactly ultrahomogeneous}) if 
for every finite subset
(resp.~compact subset) $A$ of $X$, 
and for every isometric embedding 
$\phi\colon A\to X$, 
there exists an isometric bijection 
 $F\colon X\to X$ such taht $F|_{A}=\phi$. 
 Remark that the usual  Urysohn universal metric space and the separable $\youfin(R)$-injective complete 
 $R$-ultrametric spaces are ultrahomogeneous
 and compactly ultrahomogeneous, 
 which are also proven by 
 the back-and-forth argument. 
 For the ultrahomogeneity,  see \cite[Theorem 3.2]{MR2435148} and \cite[Proposition 2.7]{MR2754373}. 
 For the compact ultrahomogeneity, see 
 \cite[Corollary 6.18]{MR2754373}, 
 \cite{MR0072454}, 
\cite{MR1781067}, 
and  
\cite[Subsection 4.5]{MR2435148}.

Our second result 
sates  that 
all 
petaloid 
 spaces satisfy
 the 
 compact 
ultrahomogeneity despite their  non-separability.

\begin{thm}\label{thm:222}
For every  uncountable 
range set $R$,  
every $R$-petaloid ultrametric space is 
compactly ultrahomogeneous. 
Consequently, it is ultrahomogeneous. 
\end{thm}

The proofs of 
 Theorems \ref{thm:main1} and 
\ref{thm:222} will be presented in 
Section \ref{sec:uandh}, 
and  both of them are  corollaries  
of 
Theorem  \ref{thm:petoidext} asserting 
that for every uncountable range set $R$, 
for every
 $S\in \yoofam{R}$, and
 for all  $R$-petaloid spaces 
 $(X, d)$ and $(Y, e)$, every isometric bijection 
 $f\colon \yopetal{X}{S}\to \yopetal{Y}{S}$ can be 
 extended to  an isometric bijection 
 $F\colon X\to Y$. 
In  the proof of Theorem  \ref{thm:petoidext},  to obtain an isometric bijection 
$F\colon X \to Y$, 
we construct isometric bijections  between 
$\yopetal{X}{T}$ and $\yopetal{Y}{T}$ for all 
$T\in \yoofam{R}$ such that 
$T\setminus S$ is finite using the 
back-and-forth argument,  
 and we glue them together by the transfinite induction.

In the next theorem, 
we prove the existence of petaloid spaces and 
we show that 
 injective spaces 
constructed  in 
\cite{Ishiki2023const}, 
\cite{MR4282005}, and 
\cite{MR2754373}
become naturally  petaloid. 
The proof  and 
the precise 
definitions of  the spaces appearing below will be given in 
Section \ref{sec:examples}. 

\begin{thm}\label{thm:petaloidexam}
For every uncountable range set $R$, 
all the following spaces are $R$-petaloid. 

\begin{enumerate}[label=\textup{(\arabic*)}]
\item\label{item:sp:maps:th}
The 
space $(\yomaps{\yocantorc}{R}, \yomaindis)$ of 
all continuous functions $f$ from the 
Cantor set $\yocantorc$ into the space $R$ equipped with 
the nearly discrete topology such that $0\in f(X)$. 

\item\label{item:sp:met:th}
The ultrametric space
$(\yocmet{X}{R}, \umetdis_{X}^{R})$ 
of 
all $R$-valued  continuous 
pseudo-ultrametrics on a 
compact ultrametrizable  space  $X$
with an accumulation point.

\item\label{item:sp:gh:th}
The non-Archimedean Gromov--Hausdorff space 
$(\yonghsp{R}, \yonghdis)$ 
associated with $R$.

\item\label{item:sp:cofunc:th}
The space 
$(\yomapsco{R}{\yotomega}, \yomaindisco)$
of all maps $f\colon R\to \omega_{0}$ 
for which $f(0)=0$ and 
the  support of $f$ are tenuous, 
where $\omega_{0}$ is the set of all natural numbers. 
\end{enumerate}
\end{thm}

\begin{rmk}
Intriguingly, 
it is unknown 
 whether 
petaloid ultrametric spaces 
can be obtained by the
classical constructions, 
i.e.,  the method using 
the Kat\v{e}tov function spaces,  and 
the Urysohn-type amalgamation 
(the  way of  the  Fra\"{i}ss\'{e} limit). 
\end{rmk}

\begin{rmk}
Theorem 
\ref{thm:petaloidexam}
can be 
considered as 
a partial answer of 
\cite[Question 3.1]{Z2005}, 
which 
asks to  describe  the topology of 
the non-Archimedean Gromov--Hausdorff space. 
\end{rmk}

\section{Preliminaries}\label{sec:pre}
Before proving the existence of 
petaloid spaces, 
we axiomatically 
discuss their basic properties  in this section,
and 
the assertions herein  are
prepared to support the proof of  Theorems \ref{thm:main1} and 
\ref{thm:222}. 
To readers who are wondering if there exists a petaloid space, 
the author kindly  recommends
reading 
the proof of Theorem 
\ref{thm:petaloidexam} first, 
which is independent from 
this section and 
 Theorems \ref{thm:main1} and 
\ref{thm:222}. 
\begin{lem}\label{lem:000}
Let $R$ be an uncountable range set, and 
$(X, d)$ be an $R$-petaloid ultrametric space. 
Then the set $\yopetal{X}{\{0\}}$ is a 
singleton. 
\end{lem}
\begin{proof}
The property \ref{item:pr:sep} implies that 
$\yopetal{X}{\{0\}}$ is the $\{0\}$-Urysohn universal 
ultrametric space. 
Since every $\{0\}$-valued ultrametric space must be 
isometric to 
the one-point space, 
so is the space 
$\yopetal{X}{\{0\}}$. 
\end{proof}

%statebody
\begin{lem}\label{lem:18:p6}
Let $R$ be an uncountable  range set, 
and $(X, d)$ be the $R$-petaloid ultrametric space. 
Then 
for every $S\in \yoofam{R}$,  and 
for every $x\in X$, there exists 
$p\in \yopetal{X}{S}$ 
such that $d(x, \yopetal{X}{S})=d(x, p)$. 
\end{lem}
%proof
\begin{proof}
The lemma is  deduced from  
\cite[Propositions  20.2 and  21.1]{MR2444734}
and the facts that 
$\yopetal{X}{S}$ 
$S$-valued and $S$ is tenuous
(see the property \ref{item:pr:sep}). 
\end{proof}

\begin{lem}\label{lem:petalinclusion}
Let $R$ be an uncountable range set, 
and $(X, d)$ be an $R$-petaloid space. 
If $S, T\in\yoofam{R}$ satisfy $S\yosub T$, then 
$\yopetal{X}{S}\yosub \yopetal{X}{T}$. 
\end{lem}
\begin{proof}

From 
the properties  
\ref{item:pr:cap} or \ref{item:pr:distance}, 
we deduce the 
lemma. 
\end{proof}

The proof of the next lemma is presented in 
\cite[Lemma 2.12]{Ishiki2023const}. 
\begin{lem}\label{lem:discom}
Let $K$ be a subset of $[0, \infty)$. 
Then 
$K$ is tenuous if and only if 
$K$ is a closed subset of $[0, \infty)$ with respect to the Euclidean topology 
 and satisfies that 
  $K\cap [r, \infty)$ is finite
  for all $r\in (0, \infty)$. 
  \end{lem}
  
By the definition of the tenuous sets, 
we obtain: 
\begin{lem}\label{lem:noref}
If  $\{S_{i}\}_{i=0}^{k}$  is a finite family of 
tenuous subset of $[0, \infty)$, 
then the unison $\bigcup_{i=0}^{k}S_{i}$ is 
also tenuous. 
\end{lem}
In what follows, 
without referring to Lemma 
\ref{lem:noref}, 
we use the property that 
the  union of finite many 
tenuous sets is tenuous.

We now discuss 
the completeness of petaloid spaces. 
Let $R$ be an uncountable  range set, 
and  $(X, d)$ be 
an $R$-petaloid ultrametric space. 
For every $x\in X$, 
we define the 
\emph{trace} $\yotr(x)$ of $x$ by 
$\yotr(x)=\bigcap \{\, S\in \yoofam{R}\mid x\in \yopetalq{S}\, \}$. 
Note that $\yotr(x)\in \yoofam{R}$ and  
 $\yotr(x)\yosub S$ whenever $x\in \yopetalq{S}$. 

\begin{lem}\label{lem:trace}
Let $R$ be an uncountable range set, 
and $(X, d)$ be an $R$-petaloid ultrametric space. 
Then 
the following are true: 
\begin{enumerate}[label=\textup{(\arabic*)}]

\item\label{item:tr1}
For every $x\in X$, and for every $r\in (0, \infty)$, 
there exists $S\in \yoofam{R}$ such that 
$\yotr(x)\cap (r, \infty)=S\cap (r, \infty)$ and 
$x\in \yopetalq{S}$.

\item\label{item:tr2}
For every $x\in X$, we have 
$x\in \yopetalq{\yotr(x)}$. 

\end{enumerate}
\end{lem}
\begin{proof}
First we show \ref{item:tr1}. 
Take $S_{0}\in \yoofam{R}$ such that 
$x\in \yopetalq{S_{0}}$. 
Since 
the sets
$S_{0}\cap (r, \infty)$ and 
 $\yotr(x)\cap (r, \infty)$ 
 are finite
(see Lemma \ref{lem:discom}), 
the set 
$S_{0}\cap (r, \infty)\setminus \yotr(x)\cap (r, \infty)$
consists of finite many points, say 
$p_{1}, \dots, p_{n}$. 
Since each $p_{i}$ does not belong to $\yotr(x)$, 
by the definition of $\yotr(x)$, 
we can take $S_{i}\in \yoofam{R}$ such that 
$x\in \yopetalq{S_{i}}$
and $p_{i}\not \in S_{i}$. 
Put $S=\bigcap_{i=0}^{k}S_{i}$. 
Then $S$ belongs to $\yoofam{R}$ and 
satisfies 
$S\cap(r, \infty)=\yotr(x)\cap (r, \infty)$. 
According to 
the property 
\ref{item:pr:cap}, 
the set $S$ satisfies $x\in \yopetalq{S}$. 
Thus $S$ is as desired.

Second we verify \ref{item:tr2}. 
For every $i\in \zz_{\ge 0}$, 
put $t_{i}=2^{-i}$. 
According to the statement \ref{item:tr1}, 
we can find 
a family  $\{S_{i}\}_{i\in \zz_{\ge 0}}$ of tenuous range subsets of $R$
such that
 $\bigcap_{i\in \zz_{\ge 0}}S_{i}=\yotr(x)$, and 
for each $i\in \zz_{\ge 0}$, 
we have $x\in \yopetalq{S_{i}}$ and $S_{i}\cap (t_{i}, \infty)=\yotr(x)\cap (t_{i}, \infty)$. 
Combining
the property \ref{item:pr:distance} and 
 the fact that $x\in \yopetalq{S_{i}}$, 
we obtain 
 $d(x, \yopetalq{\yotr(x)})\in S_{i}\setminus \yotr(x)$. 
Since  
$S_{i}\cap (t_{i}, \infty)=\yotr(x)\cap (t_{i}, \infty)$, 
every point $q\in S_{i}\setminus \yotr(x)$ must satisfy 
$q\le t_{i}$. 
In particular we deduce that 
 $d(x, \yopetalq{\yotr(x)})\le t_{i}$. 
Since $t_{i}\to 0$
as $i\to 0$, 
the point $x$ belongs to 
the closure of 
 $\yopetalq{\yotr(x)}$. 
 Remark that 
$\yopetalq{\yotr(x)}$ is closed in $X$
since it is complete (see property \ref{item:pr:sep}). 
Therefore we conclude that $x\in \yopetalq{\yotr(x)}$. 
\end{proof}

Let us observe  the relationship between the traces and the distances. 
\begin{lem}\label{lem:approx}
Let $R$ be an uncountable range set, 
and $(X, d)$ be an $R$-petaloid ultrametric space. 
If $S\in \yoofam{R}$
and 
$x\not \in \yopetalq{S}$, 
then we obtain 
$d(x, \yopetalq{S})=\min\{\, t\in
R
\mid 
\yotr(x)\cap (t, \infty)\yosub S\cap (t, \infty)\, \}$. 
\end{lem}
\begin{proof}
Put 
$u=d(x, \yopetalq{S})$. 
By
the property \ref{item:pr:distance}
and the assumption that  $x\not\in \yopetalq{S}$, 
it is true that 
$u\in \yotr(x)\setminus S$. 
For the sake of contradiction, 
suppose that
$\yotr(x)\cap (u, \infty)\not\yosub S\cap (u, \infty)$. 
Then 
 there exists $s\in \yotr(x)\cap (u, \infty)$ 
 such that $s\not\in S$. 
 In particular, we have $u<s$. 
Put $A=(\yotr(x)\setminus \{s\})\cup S$. 
Then $A\in \yoofam{R}$. 
From $s\in \yotr(x)$ and the definition of $\yotr(x)$, 
it follows that 
$x\not\in \yopetalq{A}$.  
Put $v=d(x, \yopetalq{A})$. 
Since 
$\yopetalq{S}\yosub \yopetalq{A}$ 
(recall $S\yosub A$ and see  Lemma \ref{lem:petalinclusion}), 
we obtain 
$v\le u$. 
The property  \ref{item:pr:distance} yields 
 $v\in \yotr(x)\setminus A(=\{s\})$, 
 and hence 
$s\le u$. 
This is a  contradiction to 
 $u<s$. 
Thus we have 
$\yotr(x)\cap (u, \infty)\yosub S\cap (u, \infty)$. 
The minimality of 
 $u$ follows from $u\in \yotr(x)\setminus S$ and 
$\yotr(x)\cap (u, \infty)\yosub S\cap (u, \infty)$. 
\end{proof}

\begin{cor}\label{cor:equal}
Let $R$ be an uncountable range set, 
and $(X, d)$ be an $R$-petaloid ultrametric space. 
If $x, y\in X$ and $w\in R$ satisfies $x\neq y$ 
and $d(x, y)\le w$, 
then we have 
$\yotr(x)\cap (w, \infty)=\yotr(y)\cap (w, \infty)$. 
\end{cor}
\begin{proof}
The assumption $d(x, y)\le w$ shows that 
$d(x, \yopetalq{\yotr(y)})\le w$. 
Then Lemma \ref{lem:approx} implies 
$\yotr(x)\cap (w, \infty)\yosub \yotr(y)\cap (w, \infty)$. 
Replacing the role of $x$ with that of $y$, 
we also have 
$\yotr(y)\cap (w, \infty)\yosub \yotr(x)\cap (w, \infty)$.
This completes the proof. 
\end{proof}

\begin{lem}\label{lem:unionrange}
Let $\{S_{i}\}_{i\in \zz_{\ge 0}}$ be a
family of tenuous range sets. 
If there exists a sequence 
$\{t_{i}\}_{i\in \zz_{\ge 0}}$ in $(0, \infty)$ such that 
for all $i\in \zz_{\ge 0}$, 
we have
$t_{i+1}\le t_{i}$ and 
$S_{i}\cap (t_{i}, \infty)=S_{i+1}\cap (t_{i}, \infty)$,
and  
$t_{i}\to 0$ as $i\to \infty$, 
then the union 
$\bigcup_{i\in \zz_{\ge 0}} S_{i}$ is 
also a 
tenuous range set. 
\end{lem}
\begin{proof}
Put $S=\bigcup_{i\in \zz_{\ge 0}} S_{i}$. 
 For every $r\in (0, \infty)$, 
 take $N\in \zz_{\ge 0}$ with 
 $t_{N}<r$ and $t_{N+1}<t_{N}$. 
In this setting, 
by induction, we notice that 
if $n\in \zz_{\ge 0}$ satisfies  $n\ge N$, 
then $S_{n}\cap [r, \infty)=S_{N}\cap [r, \infty)$. 
Thus we obtain 
$S\cap [r, \infty)= 
\bigcup_{i=0}^{N}(S_{i}\cap [r, \infty))$. 
Due to
Lemma \ref{lem:discom},  each  $S_{i}\cap [r, \infty)$ is finite, 
and hence so is $S\cap [r, \infty)$. 
By the properties  that $0\in S$ and  $S_{i}\cap [r, \infty)$ is finite for all $r\in (0, \infty)$, 
we see that 
the set  $S$ is closed in $[0, \infty)$
with respect to the Euclidean topology. 
Using 
Lemma \ref{lem:discom} again,  
we conclude that the set $S$ is tenuous. 
\end{proof}

For a metric space $(X, d)$, 
for $\epsilon\in (0, \infty)$, 
 a subset $A$ of $X$ is 
 said to be an \emph{$\epsilon$-net of $X$} if 
 $A$ is finite and satisfies 
 that for all $x\in X$, there exists 
 $a\in A$ with $d(x, a)\le \epsilon$. 
 We say that a metric space $(X, d)$ is 
 \emph{totally bounded} if 
 for every $\epsilon \in (0, \infty)$, 
 there exists an $\epsilon$-net of $X$. 

\begin{prop}\label{prop:compacttrace}
Let $R$ be an uncountable range set, 
and $(X, d)$ be an $R$-petaloid ultrametric space. 
If $K$ is a totally bounded  subset of $X$, 
then there exists $S\in \yoofam{R}$ such that 
$K\yosub \yopetalq{S}$. 
\end{prop}
\begin{proof}
For each $n\in \zz_{\ge 0}$, 
let $L_{n}$ be a
$(2^{-n})$-net of $K$. 
We may assume that 
$L_{n}\yosub L_{n+1}$ for all 
$n\in \zz_{\ge 0}$. 
Put $T_{n}=\bigcup_{x\in L_{n}}\yotr(x)$. 
Then, since $L_{n}$ is finite, we have  $T_{n}\in \yoofam{R}$ and 
$T_{n}\yosub T_{n+1}$ for all $n\in \zz_{\ge 0}$. 
Since $L_{n}$ is a $(2^{-n})$-net of $L_{n+1}$, 
for each $x\in L_{n+1}$, there exists $a\in L_{n}$ with 
$d(x, a)\le 2^{-n}$. 
According to 
Corollary \ref{cor:equal}, 
we have 
$\yotr(x)\cap (2^{-n}, \infty)=
\yotr(a)\cap (2^{-n}, \infty)$. 
Thus $T_{n}\cap (2^{-n}, \infty)
=T_{n+1}\cap (2^{-n}, \infty)$
for all $n\in \zz_{\ge 0}$. 
Put $S=\bigcup_{n\in \zz_{\ge 0}}T_{n}$. 
Therefore Lemma 
\ref{lem:discom} shows that 
the set $S$ 
belongs to $\yoofam{R}$. 
From  the definition of the traces, 
it follows that 
$\bigcup_{n\in \zz_{\ge 0}}L_{n}\yosub \yopetalq{S}$. 
Since $\bigcup_{n\in \zz_{\ge 0}}L_{n}$ is dense in 
$K$ and 
$\yopetal{X}{S}$ is complete (see \ref{item:pr:sep}),  
we   conclude that 
$K\yosub \yopetalq{S}$. 
\end{proof}

We now prove that all petaloid spaces 
are complete. 
\begin{prop}\label{prop:petalcomp}
Let $R$ be an uncountable range set, 
and $(X, d)$ be an $R$-petaloid ultrametric space. 
Then $(X, d)$ is complete. 
\end{prop}
\begin{proof}
Let $\{x_{i}\}_{i\in \zz_{\ge 0}}$ be a 
Cauchy sequence in $(X, d)$. 
Then the set 
$K=\{\, x_{i}\mid i\in \zz_{\ge 0}\, \}$ is 
totally bounded, 
and  Proposition \ref{prop:compacttrace}
enables us to take $S\in \yoofam{R}$ such that 
$K\yosub \yopetalq{S}$. 
Since $\yopetalq{S}$ is complete, 
 the sequence 
$\{x_{i}\}_{i\in \zz_{\ge 0}}$ has a limit. 
Thus $(X, d)$ itself  is 
complete. 
\end{proof}

Let us confirm that 
all $R$-petaloid spaces are  injective  for all finite $R$-ultrametric spaces. 
\begin{prop}\label{prop:fininj}
Let $R$ be an uncountable range set. 
Then all 
 $R$-petaloid spaces are
$\youfin(R)$-injective. 
\end{prop}
\begin{proof}
Let $(X, d)$ be an arbitrary $R$-petaloid space. 
Take $(A, e)$ and $(B, e)$ be  finite metric spaces in 
$\youfin(R)$ with 
$A\yosub B$.
Take an isometric embedding 
$\phi\colon A\to X$. 
Put $S=e(B\times B)\cup \bigcup_{x\in A}\yotr(\phi(x))$. 
Since $B$ is finite,  
we notice that $S\in \yoofam{R}$ and 
$(B, e)$ belongs to $\youfin(S)$. 
By the definition of the trace, 
we also have $\phi(A)\yosub\yopetalq{S}$. 
The property  \ref{item:pr:sep} implies that
$\yopetalq{S}$ is 
$\youfin(S)$-injective, 
and hence there exists an isometric embedding 
$F\colon B\to \yopetalq{S}$ such that $F|_{B}=\phi$. 
Therefore 
we conclude that $(X, d)$ is $\youfin(R)$-injective. 
\end{proof}

Consequently, we obtain the next proposition. 
%statebody
\begin{prop}\label{prop:universal}
Let $R$ be an uncountable range set, 
and $(X, d)$ be an $R$-petaloid ultrametric space. 
Then every separable $R$-valued  ultrametric space 
can be isometrically embedded into 
$(X, d)$. 
\end{prop}
\begin{proof}
The $\youfin(R)$-injectivity
and the completeness  of $(X, d)$  
imply the proposition 
using induction. 
\end{proof}

For range sets $R$ and $S$ with $S\yosub R$, 
we denote by 
$\yoofamqq{R}{S}$ the set of all $T\in \yoofam{R}$
such that $S\yosub T$ and  $T\setminus S$ is finite. 
The following lemma plays an important role in 
the proofs of our main theorems. 
The proof is deduced from 
Lemma \ref{lem:approx} and 
the property \ref{item:pr:cup}. 
\begin{lem}\label{lem:dense}
If $R$ is  an uncountable range set, 
$S\in \yoofam{R}$, 
and $(X, d)$ is  an $R$-petaloid ultrametric space, 
then $\bigcup_{T\in \yoofamqq{R}{S}}\yopetalq{T}$ is dense in $X$. 
\end{lem}
\begin{proof}
Take $x\in X$. 
For every $n\in \zz_{\ge 0}$,  we put 
$T_{n}=S\cup (\yotr(x)\cap [2^{-n}, \infty))$. 
Then $T_{n}\in\yoofamqq{R}{S}$. 
From  the property \ref{item:pr:distance}, 
it follows that 
$d(x, \yopetalq{T_{n}})\le 2^{-n}$
for all $n\in \zz_{\ge 0}$. 
Thus, the set  $\bigcup_{T\in \yoofamqq{R}{S}}\yopetalq{T}$ is 
dense in $X$. 
\end{proof}

\section{Proofs of uniqueness and homogeneity}\label{sec:uandh}
The following lemma is  
a generalization of 
the injectivity. 
\begin{lem}\label{lem:extpt}
Let $R$ be an uncountable range set,  and 
$(X, d)$ and $(Y, e)$ be  $R$-petaloid ultrametric spaces. 
Let  
$k\in \zz_{\ge 0}$, 
$T_{0}, \dots, T_{k}\in \yoofam{R}$ and 
$S\in \yoofam{R}$ such that 
$\bigcup_{i=0}^{k}T_{i}\yosub S$. 
Let 
$A\sqcup \{\omega\}$ be a finite subset  of 
$\yopetal{X}{S}$ and 
$B$ be a finite subset of 
$\yopetal{Y}{S}$. 
Put  $G=\bigcup_{i=0}^{k}\yopetal{X}{T_{i}}$ and 
$H=\bigcup_{i=0}^{k}\yopetal{Y}{T_{i}}$. 
If  $f\colon G\sqcup A\to H\cup B$ is  an isometric 
bijection 
such that $f(G)=H$ and $f(A)=B$, 
then there exists 
$\theta\in \yopetal{Y}{S}$ 
for which
$d(f(x), \theta)=d(x, \omega)$ for all 
$x\in G\cup A$. 
Namely, 
we obtain an isometric bijection 
$F\colon G\cup A\cup \{\omega\}\to H\cup B\cup \{\theta\} $ such that $F|_{G\cup A}=f$. 
\end{lem}
\begin{proof}
If $\omega\in G$, then putting  $\theta=f(\omega)$ proves the lemma. 
We may assume that $\omega\not \in G$. 
For each $i\in \{0, \dots, k\}$, 
Lemma \ref{lem:18:p6}
enables us to 
take $p_{i}\in \yopetal{X}{T_{i}}$ such that 
$d(\omega, \yopetal{X}{T_{i}})=d(\omega, p_{i})$. 
Put
$P=\{p_{0}, \dots, p_{k}\}$, 
$Q=\{f(p_{0}), \dots, f(p_{k})\}$,  and 
$\phi=f|_{A\cup P}$. 
Then $\phi\colon A\cup P\to \yopetal{Y}{S}$ is an 
isometric embedding. 
Due to  the property \ref{item:pr:sep} for $(Y, e)$,  we can apply 
the $\youfin(R)$-injectivity of $\yopetal{Y}{S}$ to 
$\phi$, 
$A\cup P$, and  $(A\cup P)\cup\{\omega\}$.
Then we obtain 
 $\theta\in \yopetal{Y}{S}$ such that:
 \begin{enumerate}[label=\textup{(p)}]
 \item\label{item:ppp}
 We have 
$e(f(x), \theta)=d(x, \omega)$
for all $x\in P\cup A$. 
\end{enumerate}

It remains to   prove the following statement. 
\begin{enumerate}[label=\textup{(g)}]
\item\label{item:qqq}
For every 
 $x\in G\cup A$, 
we have
$e(f(x), \theta)=d(x, \omega)$. 

\end{enumerate}
Take an arbitrary point  $x\in G\cup A$. 
If $x\in A\cup P$, then 
the fact  
\ref{item:ppp} means that 
$e(f(x), \theta)=d(x, \omega)$. 
If $x\not \in A\cup P$, then 
we can take $j\in \{0, \dots, k\}$ such that $x\in \yopetal{X}{T_{j}}$. 
Put $c=d(\omega, \yopetal{X}{T_{j}})(=d(\omega, p_{j}))$. 
Then the fact \ref{item:ppp} implies  $e(\theta,f(p_{j}))=c$. 
The assumption that $\omega\not \in G$ shows 
 $c>0$. 
The property \ref{item:pr:distance} implies  
that
$c\not \in T_{j}$. 
In particular, we obtain 
 $e(f(x), f(p_{j}))\neq c$. 
 Recall that $e(f(x), f(p_{j}))=d(x, p_{j})$.  
 We divide the proof of the statement \ref{item:qqq} into two parts. 

 \yocase{1}{$e(f(x), f(p_{j}))<c$}
 In this case, 
$e(f(x), f(p_{j}))<e(\theta, f(p_{j}))$ and $d(x, p_{j})<d(\omega, p_{j})$. 
The strong triangle inequality shows that  
$e(\theta, f(x))=e(\theta, f(p_{j}))$
and $d(\omega, x)=d(\omega, p_{j})$. 
Since 
$e(\theta, f(p_{j}))=
d(\omega, x)$ (see the fact \ref{item:ppp}), 
we  obtain 
 $e(f(x), \theta)=d(x, \omega)$. 

\yocase{2}{$c<e(f(x), f(p_{j}))$}
Since $c=e(\theta, f(p_{j}))=d(\omega, p_{j})$, 
we have 
$e(\theta, f(p_{j}))<e(f(x), f(p_{j}))$ and 
$d(\omega, p_{j})<d(x, p_{j})$. 
Using the strong triangle inequality again, we also have 
 $e(f(x), f(p_{j}))=e(f(x), \theta)$ and 
$d(x, p_{j})=d(x, \omega)$. 
Since $f$ is isometric, 
 we obtain $e(f(x), f(p_{j}))=d(x, p_{j})$, 
 and hence  $e(f(x), \theta)=d(x, \omega)$. 
This finishes the proof. 
\end{proof}

\begin{lem}\label{lem:isopetapeta}
Let $R$ be an uncountable range set,  and 
$(X, d)$ and $(Y, e)$ be  $R$-petaloid ultrametric spaces. 
Let $k\in \zz_{\ge 0}$, 
$T_{0}, \dots, T_{k}\in \yoofam{R}$
and $S\in \yoofam{R}$ such that 
$\bigcup_{i=0}^{k}T_{i}\yosub S$. 
For each $i\in \{0, \dots, k\}$, let 
$g_{i}\colon \yopetal{X}{T_{i}}\to \yopetal{Y}{T_{i}}$ be 
an  isometric bijection. 
Assume that the following condition is satisfied:
\begin{enumerate}[label=\textup{(G)}]
\item\label{item:32:coh}
For all $i, j\in \{0, \dots, k\}$, 
we have $g_{i}(x)
=g_{j}(x)$ for all 
$x\in \yopetal{X}{T_{i}\cap T_{j}}$. 
\end{enumerate}
Then we can glue $\{g_{i}\}_{i=0}^{k}$ together and extend it, i.e., 
there exists an isometric
bijection  
$g\colon \yopetal{X}{S}\to 
\yopetal{Y}{S}$ such that 
$g|_{\yopetal{X}{T_{i}}}=g_{i}$ for all 
$i\in \{0, \dots, k\}$. 
\end{lem}
\begin{proof}
Put 
$G=\bigcup_{i=0}^{k}\yopetal{X}{T_{i}}$ and
$H=\bigcup_{i=0}^{k}\yopetal{Y}{T_{i}}$. 
We notice that 
$\yopetal{X}{S}\setminus G$, 
and 
$\yopetal{Y}{S}\setminus H$ are 
separable and 
take a countable dense subsets 
$A=\{a_{i}\}_{i\in \zz_{\ge 1}}$ and 
$B=\{b_{i}\}_{i\in \zz_{\ge 1}}$ of 
$\yopetal{X}{S}\setminus G$, 
and 
$\yopetal{Y}{S}\setminus H$, respectively. 
We put $A_{i}=\{a_{1}, \dots, a_{i}\}$ and 
$B_{i}=\{b_{1}, \dots, b_{i}\}$
for all $i\in \zz_{\ge 1}$. 
We also put $A_{0}=B_{0}=\emptyset$. 
Define a map $u\colon G\to H$ by 
$u(x)=g_{i}(x)$ if $x\in \yopetal{X}{T_{i}}$. 
By the assumption \ref{item:32:coh}, 
the map
$u$ is well-defined. 
To construct 
an isometric bijection $g$, 
we use the  back-and-forth argument, in which we  repeat the two types of operations alternately; in  the first operation, we extend the domain of an isometric map,   and in the second operation, we extend the codomain  using the inverse map. 
Namely, we 
will
construct  a sequence  
$\{w_{i}\colon P_{i}\to Q_{i}\}_{i\in \zz_{\ge 0}}$ of maps such that: 
\begin{enumerate}[label=\textup{(W\arabic*)}]
\item\label{item:ww:ba}
 for every $i\in \zz_{\ge 0}$, 
$P_{i}$ and $Q_{i}$ are subsets of 
$\yopetal{X}{S}$ and $\yopetal{Y}{S}$, 
respectively,  satisfying that 
$G\cup A_{i}\yosub P_{i}$ and 
$H\cup B_{i}\yosub Q_{i}$;
\item\label{item:ww:isom}
each $w_{i}$ is an isometric bijection;
\item\label{item:ww:ext}
for every $i\in \zz_{\ge 0}$, 
we have $w_{i}|_{G}=u$;
\item\label{item:ww:gg}
for every $i\in \zz_{\ge 0}$, 
we have 
$P_{i}\yosub P_{i+1}$, $Q_{i}\yosub Q_{i+1}$, 
and $w_{i+1}|_{P_{i}}=w_{i}$. 
\end{enumerate}
First, 
we define 
$P_{0}=G$, 
$Q_{0}=H$, and 
$w_{0}=u$. 
Next fix $k\in \zz_{\ge 0}$ and 
assume that we have already obtained  
$P_{k}$, 
$Q_{k}$, 
and 
 an 
isometric bijection 
$w_{k}\colon 
P_{k}
\to Q_{k}$ such that 
$w_{k}|_{G}=u$
and
$G\cup A_{i}\yosub P_{i}$ and 
$H\cup B_{i}\yosub Q_{i}$. 
We now define 
$v_{k+1}\colon P_{k}\cup \{a_{k+1}\}
\to Q_{k}\cup \{v_{k+1}(a_{k+1})\}$
as an extended map of $w_{k}$
stated in 
 Lemma \ref{lem:extpt}. 
Similarly, applying  Lemma \ref{lem:extpt} to 
$v_{k+1}^{-1}\colon Q_{k}\cup \{v_{k+1}(a_{k+1})\}
\to P_{k}\cup \{a_{k+1}\}$, 
we obtain 
an isometric map  
$m_{k+1}\colon Q_{k}\cup \{v_{k+1}(a_{k+1})\} \cup 
\{b_{k+1}\}\to 
P_{k}\cup \{a_{k+1}\}\cup \{m_{k+1}(b_{k+1})\}$. 
Then we define $w_{k+1}=m_{k+1}^{-1}$, 
$P_{k+1}=P_{k}\cup \{a_{k+1}\}\cup \{m_{k+1}(b_{k+1})\}$
and 
$Q_{k+1}=Q_{k}\cup \{v_{k+1}(a_{k+1})\} \cup 
\{b_{k+1}\}$. 

Therefore we obtain  sequences
 $\{w_{i}\}_{i\in \zz_{\ge 0}}$, 
$\{P_{i}\}_{i\in \zz_{\ge 0}}$,  
and 
$\{Q_{i}\}_{i\in \zz_{\ge 0}}$ as required. 
Put $K=\bigcup_{i\in \zz_{\ge 0}}P_{i}$
and 
$L=\bigcup_{i\in \zz_{\ge 0}}Q_{i}$. 
Note that 
$K$ and 
$L$ are
dense in $\yopetal{X}{S}$ and 
$\yopetal{Y}{S}$, 
respectively. 
We define 
$h\colon K
\to L$ by 
$h(x)=w_{k}(x)$, 
where $k\in \zz_{\ge 0}$ satisfies $x\in P_{k}$. 
Due to the conditions \ref{item:ww:ba}--\ref{item:ww:gg}, 
the map $h$ is well-defined. 
Using a canonical method by Cauchy sequences
together with the completeness of 
$\yopetal{X}{S}$ and $\yopetal{Y}{S}$, 
we obtain an isometric bijection 
$g\colon \yopetal{X}{S}\to \yopetal{Y}{S}$ such that 
$g|_{G}=u$. 
Since 
$h$ is an isometric bijection, 
and since  $K$ and $L$ are dense in $\yopetal{X}{S}$ and 
$\yopetal{Y}{S}$, respectively, 
the map $g$ is an isometric bijection between $\yopetal{X}{S}$
and $\yopetal{Y}{S}$. 
This proves the lemma. 
\end{proof}

\begin{thm}\label{thm:petoidext}
Let $R$ be an uncountable range set, 
and  
 $(X, d)$ and $(Y, e)$ be 
$R$-petaloid ultrametric spaces. 
Let $f\colon \yopetal{X}{S}\to \yopetal{Y}{S}$ be an 
isometric bijection, where $S\in \yoofam{R}$. 
Then there exists an 
isometric bijection 
$F\colon X\to Y$ such that 
$F|_{\yopetal{X}{S}}=f$. 
\end{thm}
\begin{proof}
Let  $\kappa$ be the cardinal of 
$\yoofamqq{R}{S}$. 
We represent 
$\yoofamqq{R}{S}=\{S_{\alpha}\}_{\alpha<\kappa}$ such that 
$R_{0}=S$ and 
$S_{\alpha}\neq S_{\beta}$ for all distinct 
 $\alpha, \beta< \kappa$. 
We shall construct a family 
$\{g_{\alpha}\colon \yopetal{X}{S_{\alpha}}\to \yopetal{Y}{S_{\alpha}}\}_{\alpha<\kappa}$
of isometric bijections
such that: 
\begin{enumerate}[label=\textup{(C)}]
\item\label{item:cocoh}
If $\alpha, \beta<\kappa$ satisfy 
$S_{\beta}\yosub S_{\alpha}$, 
then $g_{\alpha}|_{\yopetal{X}{S_{\beta}}}=g_{{\beta}}$. 
\end{enumerate}
We use transfinite induction. 
If $\alpha=0$, 
we define $g_{0}=f$. 
We next consider general steps. 
At the $\alpha$-th step ($\alpha<\kappa$), 
we 
will define not only 
$g_{\alpha}$ but also $g_{\gamma}$ 
for all $\gamma<\kappa$ for which  $S_{\gamma}\yosub S_{\alpha}$. 
Fix $\alpha<\kappa$ and assume that we have  
already 
conducted the $\beta$-step
 for all $\beta<\alpha$. 
Then now  we 
deal with  the $\alpha$-step 
  as follows. 

\yocase{1}{if for all $\beta<\alpha$, we have 
$\yopetal{X}{S_{\alpha}}
\not \yosub \yopetal{X}{S_{\beta}}$}
In this case,  
we denote by the $\yoabset$ the set of all 
$\beta<\kappa$ such that 
$g_{\beta}$ is already defined and 
$S_{\beta}\yosub S_{\alpha}$. 
Since $S_{\alpha}\setminus S$ is a finite set, 
the  $\yoabset$ is also finite. 
By the condition \ref{item:cocoh} with respect to 
all $\gamma<\kappa$ for which $g_{\gamma}$ is 
already defined, 
the set $S_{\alpha}$, 
the family $\{S_{\beta}\}_{\beta \in \yoabset}$, 
and the maps $\{g_{\beta}\}_{\beta\in \yoabset}$
satisfy the assumptions of 
Lemma \ref{lem:isopetapeta}
(especially,  the condition 
\ref{item:32:coh}). 
Then  we obtain an isometric map 
$g_{\alpha}\colon \yopetal{X}{S_{\alpha}}\to \yopetal{Y}{S_{\alpha}}$ such that 
$g_{\alpha}|_{\yopetal{X}{S_{\beta}}}=g_{\beta}$
for all $\beta\in \yoabset$. 
Moreover, 
for all $\gamma<\kappa$ with $\alpha<\gamma$ such that 
$S_{\gamma}
\yosub S_{\alpha}$, 
we define 
$g_{\gamma}=g_{\beta}|_{\yopetal{X}{S_{\gamma}}}$.

\yocase{2}{there exists $\beta<\kappa$ such that 
$\yopetal{X}{S_{\alpha}}\yosub \yopetal{X}{S_{\beta}}$}
Take a minimal ordinal  $\beta<\kappa$  such that 
$S_{\alpha}\yosub S_{\beta}$. 
Then,  at the $\beta$-th step,  the map 
$g_{\alpha}$  has been already defined by 
$g_{\alpha}=g_{\beta}|_{\yopetal{X}{S_{\alpha}}}$.

Therefore, we obtain
a family 
$\{g_{\alpha}\colon \yopetal{X}{S_{\alpha}}\to \yopetal{Y}{S_{\alpha}}\}_{\alpha<\kappa}$
of isometric bijections with 
the condition \ref{item:cocoh}. 
Put 
$\yodensea=\bigcup_{T\in \yoofamqq{R}{S}}\yopetal{X}{T}$
and $\yodenseb=\bigcup_{T\in \yoofamqq{R}{S}}\yopetal{Y}{T}$. 
We define a map 
$G\colon 
\yodensea\to 
\yodenseb$ 
by 
$G(x)=g_{\alpha}(x)$ if $x\in \yopetal{X}{S_{\alpha}}$. 
By the condition \ref{item:cocoh}, 
the map $G$ is well-defined. 
 Lemma \ref{lem:dense}
 shows that $\yodensea$ and $\yodenseb$
 are 
dense in $X$ and $Y$, respectively. 
In a canonical way using Cauchy sequences, 
owing to Proposition \ref{prop:petalcomp}, 
we can obtain $F\colon X\to Y$ 
satisfying that  
$F|_{\yodensea}=G$. 
Then $F$ is an isometric bijection as required. 
\end{proof}

Using Theorem \ref{thm:petoidext}, 
we can  prove our main results. 

\begin{proof}
[Proof of Theorem \ref{thm:main1}]
Let $R$ be an uncountable range set, and 
$(X, d)$ and $(Y, e)$ be 
 $R$-petaloid $R$-ultrametric spaces. 
Put $S=\{0\}$ and 
let $f\colon \yopetal{X}{S}\to \yopetal{X}{S}$
 be the trivial map (see 
 Lemma \ref{lem:000}). 
Applying Theorem 
\ref{thm:petoidext} to 
$S$ and $f$, 
we obtain an isometric 
bijection 
$F\colon X\to Y$. 
This finishes the proof of 
Theorem \ref{thm:main1}. 
\end{proof}

\begin{proof}[Proof of Theorem \ref{thm:222}]
Let $R$ be an uncountable 
range set, 
and $(X,  d)$ be an $R$-petaloid 
ultrametric space. 
Assume that  $A$ is 
 a 
 compact  subset of $X$ and $f\colon A\to X$ is  an isometric embedding. 
Using Proposition \ref{prop:compacttrace}, 
we take $S\in \yoofam{R}$ with 
$K\yosub \yopetal{X}{S}$ and 
$\phi(K)\yosub\yopetal{X}{S}$. 
Since 
$\yopetal{X}{S}$ is compactly ultrahomogeneous
(see \cite[Corollary 6.18]{MR2754373}), 
we obtain an isometric bijection 
$g\colon \yopetal{X}{S}\to \yopetal{X}{S}$ with 
$g|_{K}=f$. 
Theorem \ref{thm:petoidext} implies that 
there exists an isometric bijection 
$F\colon X\to X$ such that 
$F|_{\yopetal{X}{S}}=g$. 
The map $F$ is as desired, 
and hence the proof of Theorem \ref{thm:222} 
is finished. 
\end{proof}

\section{Examples}\label{sec:examples}
We first briefly review the four constructions of 
injective spaces. 

For a range set $R$, 
we 
define an ultrametric 
$\yodisdis_{R}$ on $R$ by 
$\yodisdis_{R}(x, y)= x\lor y$
if $x\neq y$; otherwise, $0$. 
We call it the \emph{nearly discrete (ultra)metric}
and call the topology generated by $\yodisdis_{R}$
the \emph{nearly discrete topology}. 
For a topological space $X$ and
a range set  $R$, we 
denote by $\yomaps{X}{R}$
the set of all continuous maps 
$f\colon X\to R$ from 
$X$ to $(R, \yodisdis_{M})$ 
 such that $0\in f(X)$. 
For $f, g\in \yomaps{X}{R}$, 
we define $\yomaindis(f, g)$ by 
the infimum of all 
$\epsilon \in (0, \infty)$ such that 
$f(x)\le g(x)\lor \epsilon$ and 
$g(x)\le f(x)\lor \epsilon$ for all 
$x\in X$.

Let $X$ be a  topological space, 
and $R$ be  a range set,  
we denote by 
$\yocmet{X}{R}$ the 
set of all 
continuous $R$-valued  pseudo-ultrametrics $d\colon X\times X\to R$ on $X$, where $X\times X$ and $R$ are equipped with the 
product topology and the Euclidean topology, respectively. 
For $d, e\in \yocmet{X}{S}$, 
we define $\umetdis_{X}^{S}(d, e)$ the 
infimum of all $\epsilon \in S$ such that 
$d(x, y)\le e(x, y)\lor \epsilon$ and 
$e(x, y)\le d(x, y)\lor \epsilon$ for all $x, y\in X$. 
For more details on 
$\yomaps{X}{R}$ and $\yocmet{X}{R}$, 
we refer the readers to \cite{Ishiki2023const}.

For a range set $R$, 
we denote by $\yonghsp{R}$ 
the set of all isometry classes of 
compact $R$-ultrametric spaces
and denote by  
$\yonghdis$ the 
non-Archimedean Gromov--Hausdorff distance on 
$\yonghsp{R}$, 
i.e., the value $\yonghdis((X, d), (Y, e))$ is the 
infimum of all $\yohdis(i(X), j(Y); Z, h)$, 
where $(Z, h)$ is an $R$-valued ultrametric space, 
 $\yohdis(*, *; Z, h)$ is the Hausdorff distance 
 associated with  $(Z, h)$,
 and  $i\colon X\to Z$ and $j\colon Y\to Z$ are isometric embedding. 
We call $(\yonghsp{R}, \yonghdis)$ the 
Gromov--Hausdorff $R$-ultrametric space
(For more details, see 
\cite{Z2005}, \cite{qiu2009geometry} and \cite{MR4282005}).

We denote by $\yotomega$ the set of all non-negative integers according to the set-theoretic notation. 
Let $R$ be a range set. 
We also denote by $\yomapsco{R}{\yotomega}$ 
the set of all function $f\colon R\to \yotomega$ such that $f(0)=0$ and 
the set 
$\{0\}\cup \{\, x\in R \mid f(x)\neq 0\, \}$ is tenuous. 
For $f, g\in \yomapsco{R}{\yotomega}$, 
we define an $R$-ultrametric $\yomaindisco$ on $\yomapsco{R}{\yotomega}$ by 
$\yomaindisco(f, g)=
\max\{\, r\in R\mid f(r)\neq g(r)\, \}$ if $f\neq g$; 
otherwise, $\yomaindisco(f, g)=0$. 
For more information, we refer the readers to 
\cite{MR2754373} and 
\cite{MR2667917}.

\begin{thm}\label{thm:univ4}

For every  range set $R$, 
all the following spaces are 
$\youfin(R)$ and 
complete. 
Moreover, if $R$ is finite or countable, 
they are separable. 
\begin{enumerate}[label=\textup{(\arabic*)}]
\item\label{item:sp:maps}
The 
space $(\yomaps{\yocantorc}{R}, \yomaindis)$ of 
all continuous functions from the 
Cantor set $\yocantorc$ into 
$(R, \yodisdis_{R})$. 

\item\label{item:sp:met}
The ultrametric space
$(\yocmet{X}{R}, \umetdis_{X}^{R})$ 
of 
all $R$-valued  continuous 
pseudo-ultrametrics on a 
compact ultrametrizable  space  $X$
with an accumulation point.

\item\label{item:sp:gh}
The  Gromov--Hausdorff $R$-ultrametric  space 
$(\yonghsp{R}, \yonghdis)$.

\item\label{item:sp:cofunc}
The $R$-ultrametric space 
$(\yomapsco{R}{\yotomega}, \yomaindisco)$. \end{enumerate}
\end{thm}
\begin{proof}
The cases  
\ref{item:sp:maps} and 
\ref{item:sp:met} are proven in 
\cite[Theorem 3.2 and Theorem 4.2]{Ishiki2023const}, respectively
For the  separable part, see  \cite[Proposition 3.5 and Proposition 4.3]{Ishiki2023const}. 
The case \ref{item:sp:gh} is proven by 
Wan \cite{MR4282005}. 
The case \ref{item:sp:cofunc}
can be found in 
\cite[Section 6]{MR2754373}, 
\cite[Proposition 11]{MR2667917} and \cite[Subsection 1.3]{Neretin2022}. 
\end{proof}

We now prove 
Theorem \ref{thm:petaloidexam}
stating 
that all the spaces described above are petaloid. 
\begin{proof}[Proof of Theorem \ref{thm:petaloidexam}]
Let $R$ be an uncountable range set,  and 
 $X$ be a compact ultrametrizable space with 
an accumulation point. 
Let $\yocantorc$ stands for  the Cantor set. 
For every $S\in \yoofam{R}$, 
we define the petals of the spaces stated in the theorem by 
$\yopetal{\yomaps{\yocantorc}{R}}{S}=\yomaps{\yocantorc}{S}$, 
$\yopetal{\yocmet{X}{R}}{S}=\yocmet{X}{S}$, 
$\yopetal{\yonghsp{R}}{S}=\yonghsp{S}$, 
and 
$\yopetal{\yomapsco{R}{\yotomega}}{S}=
\yomapsco{S}{\yotomega}$. 
Then these families
  satisfy the 
properties, 
\ref{item:pr:cup}, 
and 
\ref{item:pr:cap}. 
Theorem \ref{thm:univ4} implies 
that they enjoy
the property \ref{item:pr:sep}. 
It remains to show that 
they satisfy the condition  \ref{item:pr:distance}. 

We now consider  the case of 
$(\yomapsco{R}{\yotomega}, \yomaindisco)$. 
To verify  the properties
 \ref{item:pr:distance}, 
we need the following
claim that is deduced from the definition of $\yomaindisco$. 
\begin{clm}\label{clm:sub}
Let  $R$ be a range set.  
Take  $r\in R\setminus \{0\}$, 
and $f, g\in \yomapsco{R}{\yotomega}$. 
Then 
$r=\yomaindisco(f, g)$ if and only if 
$f(r)\neq g(r)$ and
$f(x)=g(x)$
whenever 
$r<x$.  
\end{clm}
Let $S\in \yoofam{R}$ and 
take an arbitrary member $f\in \yomapsco{R}{\yotomega}$. 
Put $T=\{0\}\cup \{\,x\in R\mid f(x)\neq 0 \, \}\in \yoofam{R}$. 
If $T\yosub S$, then 
$\yomaindisco(f, \yomapsco{R}{\yotomega})=
\yomaindisco(f, f)=0$. 
If $T\not\yosub S$, then 
let $r$ be the maximum of $T\setminus S$. 
Thus we have 
$T\cap (r, \infty)\yosub S\cap (r, \infty)$. 
We define $g\in \yomapsco{R}{\yotomega}$ by 
\[
g(x)=
\begin{cases}
f(x) & \text{if $r<x$;}\\
f(r)+1  & \text{if $x=r$;}\\
0 & \text{otherwise.}
\end{cases}
\]
Then by Claim \ref{clm:sub}, we have 
$r=\yomaindisco(f, g)$ and 
$\yomaindisco(f, g)=\yomaindisco(f, \yomapsco{R}{\yotomega})$. 
Thus $\yomaindisco(f, \yomapsco{R}{\yotomega})\in 
T\setminus S$. 
This proves the property \ref{item:pr:distance}. 

Remark that 
using the statements corresponding to Claim \ref{clm:sub} such as 
\cite[Corollary 2.17]{Ishiki2023const}, 
\cite[Corollary 2.30]{Ishiki2023const}, and 
 \cite[Theorem 5.1]{MR4462868}, 
 in a similar way as in  the case of $(\yomapsco{R}{\yotomega}, \yomaindisco)$, 
 we can prove 
that the spaces $(\yomaps{\yocantorc}{R}, \yomaindis)$, 
$(\yocmet{X}{R}, \umetdis_{X}^{R})$, and 
$(\yonghsp{R}, \yonghdis)$ satisfy 
\ref{item:pr:distance}, 
and hence they are 
$R$-petaloid. 
Since there is a small gap, we give a little   explanation of   the case of 
$(\yomaps{\yocantorc}{R}, \yomaindis)$,  for example. 
Let 
$S\in \yoofam{R}$, and 
$f\in \yomaps{\yocantorc}{R}$, 
and 
put $T=f(\yocantorc)\in \yoofam{R}$. 
If $T\yosub S$, then put $g=f$. 
If $T\not\yosub S$, then 
we take non-empty 
clopen subsets $A$ and $B$ of $X$ such that 
$A\cap B=\emptyset$ and $A\cup B=f^{-1}([0, r])$. 
We then
 define 
$g\in \yopetal{\yomaps{\yocantorc}{R}}{S}$ by 
\[
g(x)=
\begin{cases}
f(x) & \text{if $r<f(x)$;}\\
r  & \text{if $x\in A$;}\\
0 & \text{if $x\in B$}. 
\end{cases}
\]
Therefore, by \cite[Corollary 2.17]{Ishiki2023const}, 
in any case, 
we have 
$\yomaindis(f, g)=
\yomaindis(f, \yopetal{\yomaps{\yocantorc}{R}}{S})$, 
and hence $\yomaindis(f, \yopetal{\yomaps{\yocantorc}{R}}{S})\in T\setminus S$.

This completes the proof of 
Theorem \ref{thm:petaloidexam}. 
\end{proof}

\begin{ques}\label{ques:bij}
Let $R$ be a range  set, 
and $\yocantorc$ be the  Cantor set. 
In this setting, 
we see that 
$(\yonghsp{R}, \yonghdis)$
and 
 $(\yocmet{\yocantorc}{R}, \umetdis_{X}^{R})$ are isometric to each other by Theorem \ref{thm:main1}. 
Does  there exist a natural 
 isometric bijection  between 
them?
\end{ques}

\begin{ques}\label{ques:isom}
Let $R$ be an uncountable range set, 
and $(X, d)$ be an $R$-petaloid space. 
We denote by $\mathrm{Isom}(X, d)$
the group of isometric bijections from $X$ into itself. 
What is the isometry group $\mathrm{Isom}(X, d)$?. 
Namely, what interesting 
properties does $\mathrm{Isom}(X, d)$ satisfy?
\end{ques}

\begin{ques}\label{ques:weightisom}
Let  $(X, d)$ and $(Y, e)$ be 
complete 
$\youfin(R)$-injective 
$R$-valued ultrametric spaces
with the same topological weight. 
In this setting, 
 are $(X, d)$ and $(Y, e)$ isometric to each other?
 If not, is there a sufficient and necessary condition for 
 isometry?
 The author thinks that 
 we can find  counterexamples since 
there exists  
 a complete
  $\yofin$-injective metric  space  $(X, d)$
such that   $\mathrm{Isom}(X, d)$ is trivial
(see \cite{MR3665056}). 
\end{ques}

\begin{ac}
The author would like to thank the referee for 
helpful comments and suggestions. 
\end{ac}

%\nocite{*}
\bibliographystyle{myplaindoidoi}
\bibliography{../../../bibtex/UU.bib}

\end{document}